\documentclass{article}
\usepackage{graphicx}
\usepackage{amsfonts, amsmath, amssymb, amsthm, xcolor, mathtools, mleftright}
\usepackage[margin=1in]{geometry}
\usepackage[shortlabels]{enumitem}
\usepackage{hyperref, comment}

\newtheorem{thm}{Theorem}[section]
\newtheorem{lem}[thm]{Lemma}

\newtheorem{prop}[thm]{Proposition}

\newtheorem{conj}[thm]{Conjecture}

\theoremstyle{definition}

\newtheorem{examp}[thm]{Example}

\theoremstyle{remark}
\newtheorem{rem}[thm]{Remark}

\DeclareMathOperator{\tr}{tr}

\DeclareMathOperator{\disc}{disc}

\DeclareMathOperator{\Diag}{Diag}

\DeclareMathOperator{\ind}{ind}

\DeclareMathOperator{\Mat}{Mat}
\DeclareMathOperator{\Inv}{Inv}

\newcommand{\GL}{\mathrm{GL}}
\newcommand{\SL}{\mathrm{SL}}

\newcommand{\DD}{\mathrm{DD}}

\newcommand{\CC}{\mathbb{C}}
\newcommand{\FF}{\mathbb{F}}

\newcommand{\QQ}{\mathbb{Q}}

\newcommand{\RR}{\mathbb{R}}
\newcommand{\ZZ}{\mathbb{Z}}

\newcommand{\B}{\mathcal{B}}

\newcommand{\ba}{\overline}

\newcommand{\cross}{\times}
\newcommand{\tensor}{\otimes}

\renewcommand{\to}{\mathop{\rightarrow}\limits}
\let\left\mleft
\let\right\mright
\newcommand{\size}[1]{\lvert #1 \rvert}
\newcommand{\Size}[1]{\left\lvert #1 \right\rvert}
\newcommand{\norm}[1]{\lVert #1 \rVert}
\newcommand{\Norm}[1]{\left\lVert #1 \right\rVert}

\newcommand{\intsec}{\cap}

\newcommand{\nequiv}{\not\equiv}

\renewcommand{\>}{\right\rangle}
\newcommand{\ignore}[1]{}

\renewcommand{\epsilon}{\varepsilon}

\title{Galois groups of random integer matrices}
\author{Theresa C.\ Anderson and Evan M.\ O'Dorney}
\date{\today}

\begin{document}

\maketitle

\begin{abstract}
  We study the number $M_n(T)$ be the number of integer $n\times n$ matrices $A$ with entries bounded in absolute value by $T$ such that the Galois group of characteristic polynomial of $A$ is not the full symmetric group $S_n$. One knows $M_n(T) \gg T^{n^2 - n + 1} \log T$, which we conjecture is sharp. We first use the large sieve to get $M_n(T) \ll T^{n^2 - 1/2}\log T$. Using Fourier analysis and the geometric sieve, as in Bhargava's proof of van der Waerden's conjecture, we improve this bound for some classes of $A$.
\end{abstract}

\section{Introduction}
In recent years there has been a resurgence in the study of polynomials with non-generic Galois group, in particular the following noted breakthrough.

\begin{thm}[Bhargava \cite{Bhargava_vdW}, Theorem 1; conjectured by van der Waerden \cite{vdW1936}] \label{thm:vdW-Bh}
For an integer $n \geq 2$, let $E_n(T)$ be the number of monic integer polynomials
\[
  f(x) = x^n + a_1 x^{n-1} + \cdots + a_n, \quad |a_i| \leq T
\]
with Galois group strictly smaller than the full symmetric group $S_n$, where $T \to \infty$. Then
\[
  E_n(T) \asymp T^{n-1}.
\]
\end{thm}

Theorem \ref{thm:vdW-Bh} follows decades of incremental advances on the problem \cite{Knobloch1955,Gallagher1972,CD2020,6author}.
The dominant contribution comes from $f$ having a rational root. The asymptotic probability is unaffected by whether or not we require $f$ to be monic.

This breakthrough concerns the ``large box model'' where $n$ is fixed and the height $T \to \infty$. Work has also been fruitful for the complementary ``restricted coefficient model'' where the coefficients $a_i$ are drawn from a fixed set and $n \to \infty$ \cite{BV19,BSK20,BSKK23}.

Many directions of generalization are possible. In a recent preprint \cite{ABO2024}, the authors and Bertelli proved the analogue of the van der Waerden--Bhargava theorem for \emph{reciprocal} polynomials, where the generic Galois group is the wreath product $S_2 \wr S_n$. We found that the most common non-generic Galois group is the index-$2$ subgroup $(S_2 \wr S_n) \intsec A_{2n} < S_{2n}$, which also yields the best known bounds for occurrence of non-generic Galois groups in the restricted coefficient model \cite{Hokken23}.

In this paper, we pursue a different direction: we look at the Galois groups of the characteristic polynomials $\chi_A$ of random $n \times n$ matrices $A$. These were considered by Eberhard \cite{Eberhard} under a natural analogue of the restricted coefficient model, showing that if the entries $a_{ij}$ are drawn independently from a finitely supported measure on $\ZZ$ satisfying some conditions, then with high probability as $n \to \infty$, the Galois group of $\chi_A$ is either $S_n$ or $A_n$. A follow-up paper of Ferber, Jain, Sah, and Sawhney \cite{FJSS} specializes to symmetric matrices, showing again that irreducibility of $\chi_A$ holds almost surely as $n \to \infty$. We recently learned of work of Bary-Soroker, Garzoni, and Sodin \cite{2025arXiv250217218B} that shows a similar result for tridiagonal matrices.  While these papers address the restricted entry model, we think it timely to revisit the large box model. Define $M_n(T)$ to be the number of integer $n\times n$ matrices $A$ with entries in $[-T, T]$ such that the characteristic polynomial $\chi_A$ of $A$ does not have Galois group the full $S_n$. We wish to improve upon the trivial bound $M_n(T) \ll T^{n^2}$ from counting all matrices.

One way that $\chi_A$ can fail to be of type $S_n$ is if it is reducible, that is, the action of $A$ on $\QQ^n$ has an invariant subspace defined over $\QQ$. Let $R_n(T) \leq M_n(T)$ be the number of such $A$. Rivin speculated that $R_n(T) \asymp T^{n^2 - n + 1} \log T$, the lower bound being known \cite[Conjecture 12]{Rivin06pre}. For unexplained reasons, this conjecture was removed in preparing the published version of the paper \cite{Rivin08}. However, we have seen no reason to doubt Rivin's conjecture on $R_n(T)$ and propose strengthening it to $M_n(T)$:
\begin{conj}\label{conj:M_n}
    Let $M_n(T)$ be the number of integer $n\times n$ matrices $A$ such that the characteristic polynomial $\chi_A$ of $A$ does not have Galois group the full $S_n$. Then as $T \to \infty$,
    \[
        M_n(T) \asymp T^{n^2 - n + 1} \log T.
    \]
\end{conj}

Similarly to the van der Waerden--Bhargava theorem, the lower bound is known and comes from counting matrices with an integer eigenvalue, as we explain in Section \ref{sec:Zeig} below.  It is interesting to note the necessity of the $\log$ factor, which is also present in \cite{ABO2024} when counting reciprocal polynomials, but absent in the van der Waerden--Bhargava theorem.  Refining these estimates into precise asymptotic formulas would be an exciting (and lofty) goal for future work.  

\subsection{Notation}
In this paper, the norm of a matrix is the $\infty$-norm, $\norm{[a_{ij}]} = \max_{i,j} \size{a_{ij}}$. Some literature uses the $2$-norm (Euclidean norm) instead; this only affects outlying constants, which are outside the scope of this paper.

If $A$ is a matrix, we denote by $\chi_A(x)$ its characteristic polynomial, $\det(xI - A)$. If $\chi_A$ is irreducible over $\QQ$, we denote by $K_A$ the attached number field $\QQ[x]/\chi_A(x)$. We let $G_A$ be the Galois group of $\chi_A$, a subgroup of $S_n$ that permutes the roots of $\chi_A$ (transitively if and only if $\chi_A$ is irreducible).

As usual, if $f(T), g(T)$ are real-valued functions of a sufficiently large real number $T$, then the notations
    \[
    f(T) \ll g(T), \quad g(T) \gg f(T), \quad f(T) = O\bigl(g(T)\bigr)
    \]
    mean that $\size{f(T)} < c \cdot \size{g(T)}$ for sufficiently large $T$ and some constant $c$, while the notations
    \[
    f(T) \asymp g(T), \quad f(T) = \Theta\bigl(g(T)\bigr)
    \]
    mean that $f(T) \ll g(T)$ and $f(T) \gg g(T)$. Finally, $f(T) = o\bigl(g(T)\bigr)$ means that $\lim_{T \to \infty} f(T)/g(T) = 0$. The implied constants $c$ may depend on the degree $n$ but not on the height $T$; if they depend on any other variables, we will include these variables in a subscript. 
\subsection{Results}
In this paper we make progress toward proving Conjecture \ref{conj:M_n}, which appears to be very hard to prove in general. We use a number of techniques, which are effective for different classes of matrices. In Section \ref{sec:CS_HIT}, we use the Cohen--Serre irreducibility theorem (an application of the large sieve) to get a power saving upon the trivial bound $M_n(T) \ll T^{n^2}$:
\begin{thm} \label{thm:large}
  We have the bound
  \begin{equation} \label{eq:large}
    M_n(T) \ll T^{n^2 - 1/2} \log T. % )^{{\log(n!/2 - 1)} - \epsilon}.
  \end{equation}
\end{thm}

In Section \ref{sec:red1}, we use a sieve based on Fourier analysis, modeled after the work of a group including the first author \cite{6author}, to get a bound on the number of matrices $A$ such that $\chi_A$ is reducible with an irreducible factor of low degree:

\begin{thm}\label{thm:red1}
  The number $R_{n,k}(T)$ of $n\times n$ matrices $A$ with $\norm{A} \leq T$ and characteristic polynomial $\chi_A$ having an irreducible factor of degree $k \leq n/2$ is bounded by
  \[
    R_{n,k}(T) \ll T^{n^2 - \frac{n-k}{n+k-1}}.
  \]
\end{thm}
This improves on the large sieve when $k \leq n/3$. For $k \gg \sqrt{n}$, it also improves on the bound $R_{n,k}(T) \ll T^{n^2 - n + k(k+1)/2}\log T$ due to Ostafe and Shparlinski \cite[equation (3.2)]{OSNonDiag}. When $k$ is fixed and $n$ is large, we obtain a savings of $1 - O(1/n)$ over the trivial bound.

In Section \ref{sec:red_short}, we use a combinatorial argument to get the full upper bound $T^{n^2 - n + 1} \log T$ but for a rather restricted class of matrices $A$, namely those which preserve a subspace generated by vectors of low height:

\begin{thm} \label{thm:red_short}
  For integers $1 \leq k \leq n/2$, let $S_{n,k}(T) \leq R_{n,k}(T)$ be the number of $n\times n$ matrices $A$ with $\norm{A} \leq T$ satisfying the following properties:
  \begin{enumerate}[(a)]
    \item\label{short:lat} $A$ preserves a lattice $\Lambda \subset \ZZ^n$ of dimension $k$;
    \item\label{short:irr} The characteristic polynomial of $A|_{\Lambda}$ is irreducible;
    \item\label{short:short} The orthogonally complementary lattice $\Lambda^\perp$ is generated by vectors of length $\leq T$.
  \end{enumerate}
  Then
  \[
    S_{n,k}(T) \ll T^{n^2 - n + 1}\log T.
  \]
  
\end{thm}
Observe that conditions \ref{short:lat} and \ref{short:irr} force $\chi_A$ to have an irreducible factor $\chi_{A|_\Lambda}$ of degree $k$. Also, $\Lambda^\perp$ is an invariant lattice of dimension $n - k$ for the transpose $A^\top$, with $\chi_{A^\top|_{\Lambda^\perp}} \cdot \chi_{A|_\Lambda} = \chi_A$.

Finally, in Section \ref{sec:geosieve}, we use a geometric sieve modeled after Bhargava's proof of Theorem \ref{thm:vdW-Bh} to obtain a bound on the number of $A$ for which $\chi_A$ defines a primitive number field of large discriminant.  Our main new twist used in the proof is to take a double discriminant with respect to the direction of translation by a diagonal matrix with separable characteristic polynomial.  Recall that a field extension $K/\QQ$ is \emph{primitive} if it has no proper intermediate extensions $K \supsetneq L \supsetneq \QQ$; this is equivalent to the Galois group $G_A \subseteq S_n$ not lying in any of the wreath products $S_a\wr S_b \subset S_n$, where $a,b>1$ are integers with $ab = n$.

\begin{thm}\label{thm:geosieve}
  Let $L_n(T) \leq M_n(T)$ be the number of matrices $A$ with $\norm{A} \leq T$ such that $\chi_A$ is a defining polynomial for a primitive number field $K_A$ with discriminant $D \coloneqq \lvert\disc K_A \rvert \geq T^2$. Then
  \begin{equation} \label{eq:geosieve}
    L_n(T) \ll_\epsilon T^{n^2 - 1 + \epsilon}.
  \end{equation}
\end{thm}
Just as in \cite{Bhargava_vdW}, it is plausible that using Fourier analysis, one could remove the $\epsilon$ from the bound \eqref{eq:geosieve} and also extend the region of applicability to $D \geq T^{2 - \delta}$ for some small constant $\delta$. However, we do not attempt this strengthening, which is of doubtful value in the absence of a complementary result for small discriminants $D \leq T^{2 - \delta}$. In \cite{Bhargava_vdW}, the needed bound for small discriminants (Case II therein) is quickly proved by multiplying Schmidt's upper bound on the number of number fields of bounded discriminant \cite{Schmidt} by Lemke Oliver--Thorne's upper bound on the number of polynomials of bounded height cutting out a given field \cite[Proposition 2.2]{LOT20}. However, to apply such a technique in our setting would require a third factor, namely, the number $M_n(T; f)$ of matrices $A$ with a given characteristic polynomial $f$. For fixed irreducible $f$ and $T \to \infty$, Eskin--Mozes--Shah \cite{EMS96} show that
\begin{equation} \label{eq:EMS96}
  M_n(T; f) = (C_f + o_f(1)) T^{n(n-1)/2},
\end{equation}
but the constant $C_f$ and the rate of convergence depend heavily on $f$. The best uniform bound is due to Habegger, Ostafe, and Shparlinski \cite[Theorems 2.1--2.2]{ShpI}:
\[
  M_n(T; f) \ll_n T^{n^2 - n - 1 + o(1)}.
\]
Observe that this result can be applied to at most $T^{n+1}$ polynomials $f$ if we want to obtain nontrivial bounds. It is believed that \cite[Conjecture 1.1]{ShpI}
\[
  M_n(T; f) \ll_n T^{n(n-1)/2 + o(1)},
\]
but this conjecture is apparently hard. We expect that progress in the small discriminant case will require a sharper error term in \eqref{eq:EMS96} by careful analysis of the spectral and ergodic theory on the manifolds involved, inspired by the recently improved error terms  due to Blomer--Lutsko \cite{BlomerLutsko2024} for the related question of counting matrices of bounded norm in $\SL_n\ZZ$.

\textbf{Acknowledgments}: This research was supported by the NSF. We thank Alina Ostafe, Igor Shparlinski, Daniele Garzoni, and Ofir Gorodetsky for insightful comments and Robert J. Lemke Oliver for helpful conversations about this work.

\section{Integer eigenvalue and the lower bound of Conjecture \ref{conj:M_n}} \label{sec:Zeig}
In the van der Waerden--Bhargava setting, the main term was for polynomials with a rational root. Analogously, we conjecture that the dominant contribution to $M_n(T)$ comprises matrices with an integer eigenvalue. As pointed out by Rivin \cite[Corollary 11]{Rivin06pre}, the count of such matrices is a consequence of work of Katznelson \cite[Theorem 1]{Kat93}. We can easily get $T^{n^2 - n + 1}$ such matrices by setting all but the last entry in the last column to $0$, which makes $\begin{bmatrix} 0&\cdots&0&1\end{bmatrix}^\top$ an eigenvector. The true order of magnitude is only slightly larger:
\begin{thm}
    The number $R_{n,1}(T)$ of integer $n \times n$ matrices with entries bounded by $T$ and at least one integer eigenvalue is given asymptotically by
    \[
      R_{n,1}(T) \asymp T^{n^2 - n + 1} \log T.
    \]
\end{thm}
\begin{proof}
  Let $A$ be such a matrix with eigenvalue $\lambda$. Then $\size{\lambda} \leq nT$, and $B = A - \lambda I$ is singular. Counting singular matrices in an expanding box was done by Katznelson \cite[Theorem 1]{Kat93}, who showed that for any bounded convex open set $\B \subset \Mat_n(\RR)$ containing the origin, there is a constant $c_\B > 0$ for which
  \begin{equation}
    \Size{\{A \in T\B \intsec \Mat_n(\ZZ) : \det A = 0\}} = c_\B T^{n^2 - n} \log T + O(T^{n^2 - n}).
  \end{equation}
  In our case, we can simply take $\B$ to be the box $\{A \in \Mat_n(\RR) : \norm{A}_\infty \leq 1\}$ consisting of matrices whose entries are bounded in absolute value by $1$. Summing over $\lambda \ll T$, we get $R_{n,1}(T) \ll T^{n^2 - n + 1} \log T$, which is the desired upper bound.

  For the lower bound, note that if $\lambda \leq T/2$ and $\norm{B}_\infty \leq T/2$, then $\norm{A}_\infty = \norm{B + \lambda I} \leq T$. Moreover, each $A$ counted by $R_{n,1}(T)$ has at most $n$ integer eigenvalues and so is expressible in at most $n$ ways as $B + \lambda I$ with $B$ singular.
\end{proof}

\begin{rem}
  A different proof for the upper bound may be found in Shparlinski \cite[Theorem 6]{Shp09}.
\end{rem}

\section{The Cohen--Serre irreducibility theorem and Theorem \ref{thm:large}}
\label{sec:CS_HIT}

The Cohen--Serre quantitative Hilbert irreducibility theorem, proved in special cases by Cohen \cite{Coh79} and generalized by Serre \cite[\textsection 13.1; Theorem 1]{SerreMW} is a tool that applies quite generally to the Galois group of any family of polynomials whose coefficients are themselves polynomials in a family of variables. We follow the presentation of Zywina \cite[Theorem 1.2]{Zywina}, specializing to $k = \QQ$:
\begin{thm}[Cohen--Serre]
Let $F = F(x, Y_1,\ldots, Y_m) \in \QQ(Y_1,\ldots, Y_m)[x]$ be a separable polynomial of degree $n$ having Galois group $\Gamma \subseteq S_n$ over the rational function field $\QQ(Y_1,\ldots, Y_m)$. For $y = (y_1,\ldots, y_m) \in \QQ^m$, let $\Gamma_y \subseteq \Gamma$ be the Galois group of $F(x, y_1,\ldots, y_m)$ as a polynomial in $x$. Denote by $\Omega_F$ the subvariety of those $(y_1,\ldots, y_m)$ for which $\Gamma_t$ is undefined (because $f$ has a multiple root or one of its coefficients has a pole). Then as $T \to \infty$,
\[
  \Bigl|\{y \in \ZZ^m \setminus \Omega_F : \size{y_i} \leq T, \Gamma_y \neq \Gamma\}\Bigr| \ll_{m,F} T^{m-1/2} \log T.
\]
\end{thm}
In our situation, we take $Y_1,\ldots, Y_m = Y_{n^2}$ to be the entries of the matrix $A$ and $F = \chi_A$. We have $\Gamma = S_n$ since any integer polynomial can be realized as a characteristic polynomial. The exceptional set $\Omega_F$ comprises matrices with a repeated eigenvalue, of which there are at most $O(T^{n^2 - 1})$ of norm up to $T$. Hence we obtain Theorem \ref{thm:large}.

\begin{rem}
For some subgroups $\Gamma \subsetneq S_n$, the number of matrices $A$ with Galois group $\Gamma$ has a better bound $T^{n^2 - 1 + \delta(\Gamma)} \log T$, where $0 < \delta(\Gamma) < 1$ is determined combinatorially from $\Gamma$; see Zywina \cite[Theorem 1.4]{Zywina}. Unfortunately, it does not seem possible to tighten the exponent below $n^2 - 1$ with these methods.
\end{rem}

\section{Reducible matrices and Theorem \ref{thm:red1}}\label{sec:red1}
If the characteristic polynomial $\chi_A(x)$ is reducible but does not have an integer root, it has a factor $g(x)$ of degree $k$, $2 \leq k \leq n/2$. In this section we estimate the number $R_{n,k}(T)$ of matrices $A$ having this kind of decomposition.

\begin{proof}[Proof of Theorem \ref{thm:red1}]
Straightforward adaptation of Katznelson's methods leads to difficulties in which the error terms swamp the main term. Instead, we use a Fourier-analytic method inspired by recent successes in counting problems of this type \cite{6author,Bhargava_vdW_short,ABO2024}. Let $p > T$ be a large prime to be specified more precisely later. We sieve out $A$ based on the fact that, if $\chi_A(x)$ has an irreducible factor $g(x)$ of degree $k$, the same is true modulo $p$, and this factor $g$ must still have the form
\[
  g(x) = \sum_i a_i x^{k-i}, \quad a_i \ll T^i.
\]
Let $v$ be an eigenvector of $A$ whose eigenvalue $\lambda$ is a root of $g$. Then the subspace of $\CC^n$ spanned by the $k$ conjugates of $v$ is defined over $\QQ$, so we have found a subspace $W \subset \QQ^n$ on which the characteristic polynomial of $A|_W$ is $g$. Let $\Lambda = W \intsec \ZZ^n$ be the corresponding primitive lattice. Then the image $\ba\Lambda$ of $\Lambda$ in $\FF_p^n$ is a lattice such that $A|_{\ba\Lambda}$ has characteristic polynomial $\ba g = g \bmod p$. If $p \nmid \disc g$ (as holds for all but at most $O(1)$ primes $p > T$), then $\ba g$ has no repeated roots, there is a unique linear operator on $\FF_p^k$ up to conjugacy that has characteristic polynomial $\ba g$, namely the companion matrix $C_{\ba g}$. For each particular choice of conjugate $C : \ba \Lambda \to \ba \Lambda$ of the companion matrix, let $\Psi_C : \Mat_n(\FF_p^n) \to \ZZ$ be the selector that sends a matrix $A : \FF_p^n \to \FF_p^n$ to $1$ if $A|_{\ba\Lambda} = C$, and $0$ otherwise. Let $\Psi_g$ be the sum of the $\Psi_C$, over all distinct $\Lambda$ and $C$. Then for $A$ with $g \mid \chi_A$, we have shown that
\[
  \Psi_g(A) \geq 1.
\]
Let $\Phi : \Mat_n(\RR) \to \CC$ be a Schwartz bump function that is identically $1$ on matrices of norm at most $1$. Then the count $R_{n,g}$ of such matrices is at most the left-hand side of a Poisson summation:
\[
  R_{n,g} \leq \sum_{A \in \Mat_{n}(\ZZ)} \Phi\left(\frac{A}{T}\right) \Psi_g(A) = T^{n^2} \sum_{B \in \Mat_{n}(\ZZ)} \widehat{\Phi}\left(\frac{TB}{p}\right)\widehat{\Psi}_g(B) 
\]
with conventions as in \cite[\textsection 5.1]{ABO2024}. We take it that the inner product of two matrices $A$, $B$ is $\tr(AB)$.

We first compute the main term $\widehat{\Phi}(0)\widehat{\Psi}_g(0)$. To compute $\widehat{\Psi}_g(0)$, we multiply:
\begin{itemize}
  \item There are $\binom{n}{k}_p \asymp p^{kn - k^2}$ subspaces $\bar \Lambda$.
  \item For each $\bar \Lambda$, the number of $C$ is $\asymp p^{k^2 - k}$, because we can conjugate $C$ by any of the $\asymp p^{k^2}$ matrices in $\GL_k(\FF_p)$, but the centralizer of $C$ has order $\Big\lvert\big(\FF_p[X]/g(X)\big)^\cross\Big\rvert \asymp p^k$.
  \item For fixed $C$, $\Psi_C$ is supported on an $n(n-k)$-dimensional affine subspace of $\Mat_n(\FF_p)$.
\end{itemize}
Hence
\begin{align*}
  \widehat\Psi_C(0) &= p^{n(n-k) - n^2} = p^{-kn} \\
  \widehat\Psi(0) &\asymp p^{kn - k^2} \cdot p^{k^2 - k} \cdot p^{-kn} = p^{-k} \\
  \widehat\Phi(0) \widehat\Psi(0) &\asymp p^{-k}.
\end{align*}
(Note that the probability that $g \mid \chi_A$ is $p^{-k}\big(1 + O(1/p)\big)$, because the characteristic polynomials of matrices over a finite field are equidistributed modulo $p$ up to $O(1/p)$ as $p \to \infty$ for fixed $n$ by Proposition \ref{prop:char_poly_eqdst}. Thus, the loss due to those $A$ for which $\Psi_g > 1$ is of low order.)

For the error terms $\widehat{\Phi}_g(B/p)\widehat{\Psi}_g(B)$, observe that $\Phi_C$ is the selector for a coset of
\[
  \bar{\Lambda}^\bot \tensor \FF_p^n = \<u^\top v : u \in \bar{\Lambda}^\bot, v \in \FF_p^n\>.
\]
Hence the support of $\widehat{\Psi}_C$ is the orthogonal complement
\[
  \FF_p^n \tensor \bar{\Lambda}.
\]
In particular, $\widehat{\Psi}_g(B) = 0$ unless the rank $r$ of $B \in \Mat_n(\FF_p)$ satisfies $r \leq k$. The case $r = 0$ is the main term. For fixed $B$ and $r$, observe that there are $\binom{n - r}{k - r}_p \asymp p^{kn - k^2 - rn + rk}$ subspaces $\bar \Lambda$ containing the column space of $B$. The remaining factors are unchanged from the main term, so
\[
  \widehat\Phi(TB/p) \widehat\Psi(B) \ll p^{-k - rn + rk}.
\]
To sum these terms, we must estimate the number of matrices $B$ having each rank $r$ modulo $p$. We may assume that $\norm{B} \ll p/T$, as $\widehat{\Phi}(BT/p)$ drops precipitously to zero outside this range. % 
Note that $B$ is determined by $2rn - r^2$ coefficients: an $r\times r$ submatrix, nonsingular modulo $p$, and all entries in those $r$ rows and $r$ columns. Hence the number of $B$ is
\[
  \ll \left(\frac{p}{T}\right)^{2rn - r^2} \ll \left(\frac{p}{T}\right)^{2rn - r},
\]
since $p \gg T$, for a total
\begin{align*}
  R_{n,g}(T) &\ll \sum_{r = 0}^k T^{n^2 - 2rn + r} p^{rn + rk - r - k}. 
\end{align*}
At this point, we must take $p \ll T^2$ or else the $r = 1$ term will start to outpace the main term. We take $p \asymp T^{\frac{2n-1}{n+k-1}}$ to equalize all the terms and get
\[
  R_{n,g}(T) \ll T^{n^2} p^{-k}.
\]
Lastly, we sum over $g$. Note that modulo $p$, the trace-coefficient $a_1$ of $g$ is bounded by $O(T)$, whereas the other coefficients admit no bound better than $p$, so we sum over $O(T p^{k-1})$ polynomials $g$ over $\FF_p$ (and also over $O(1)$ primes $p \asymp T^{\frac{2n-1}{n+k-1}}$, to ensure at least one prime not dividing $\disc g$ for each $g$ over $\ZZ$) to bound the desired quantity
\begin{align*}
  R_{n,k}(T) &\ll T p^{k-1} \cdot T^{n^2} p^{-k} \\
  &= T^{n^2 + 1} p^{-1} \\
  &\ll T^{n^2 + 1 - \frac{2n-1}{n+k-1}}\\
  &= T^{n^2 - \frac{n-k}{n+k-1}}. \qedhere
\end{align*}
\end{proof}

\section{Reducible matrices fixing a lattice of short vectors and Theorem \ref{thm:red_short}}\label{sec:red_short}

If $\chi_A$ has an irreducible factor $g$ of degree $k$, we can find an invariant subspace $W \subseteq \QQ^n$ on which $A$ acts with characteristic polynomial $g$. One way to do this is to find a vector $v$ in the null space of $g(A)$, which has rank $\leq n - k$ and norm $\ll T^k$, and consider $W = \<v, Av, \ldots, A^{k-1}v\>$. This method in particular yields a bound on the volume of the primitive lattice $\Lambda = W \intsec \ZZ^n$: Since $\norm{g(A)} \ll T^k$, it has a kernel vector $v$ of length at most $T^{k(n-k)}$, and then $\norm{A^i v} \ll T^i \norm{v} \ll T^{k(n-k) + i}$. In particular, 
\begin{equation}
    d(\Lambda) \ll T^{k^2(n-k) + \tbinom{k}{2}}. \label{eq:Lambda_bounded_T_O1}
\end{equation}

Each lattice $\Lambda$ defines a subspace $\Inv_\Lambda \RR \subset \Mat_n \RR$ of linear maps on $\RR^n$ that fix $\Lambda$. Hence one approach to bounding $R_n(T)$ is to count lattice points of norm $\ll T$ in the integer lattice $\Inv_\Lambda \ZZ = \Inv_\Lambda \RR \intsec \Mat_n \ZZ$ and then sum over $\Lambda$. For instance, if $\Lambda = \<e_1,\ldots, e_k\>$ is the primitive lattice in a coordinate hyperplane, then
\[
  \Inv_\Lambda \ZZ = \left\{\left[
  \begin{array}{c|c}
      * & * \\ \hline
      0 & *
  \end{array}
  \right]: * \in \ZZ \right\}
\]
is the set of matrices with an $(n-k) \times k$ lower left block of zeros, and thus
\[
  R_\Lambda(T) \coloneqq \size{\Inv_\Lambda \ZZ \intsec T\B} \asymp T^{n^2 - k(n-k)},
\]
which is evidently the maximum-order contribution to $R_n(T)$ among lattices $\Lambda$, since we always have $\dim \Inv_\Lambda \ZZ = n^2 - k(n-k)$.

Let $v_1, \ldots, v_k$ be a reduced basis for $\Lambda$, and let $\ell_i = \norm{v_i}$. By convention, a reduced basis is arranged in increasing order of length: $\ell_1 \leq \cdots \leq \ell_k$. Let $G = [g_{ij}]_{i,j = 1}^k$ be the matrix describing the action of $A|_\Lambda$ with respect to the basis $\{v_i\}$. Our approach to estimating $R_\Lambda(T)$ is by two steps: we count $G$, that is, $A|_\Lambda$; then we count $A$ given $A|_\Lambda$. % In the second step, $A|_\Lambda$ determines $A$ up to translation by the lattice $\Lambda^\perp \tensor \ZZ^n$ of matrices all of whose rows annihilate $\Lambda$.
One might hope that there are $\ll T^{k^2}$-many $G$ and $\ll T^{n(n-k)}$-many $A$ for each $G$, with equality when $\Lambda$ spans a coordinate hyperplane. The bound on the number of $A$ given $G$ indeed holds, but there may be more $G$ for some lattices:

\begin{examp} \label{ex:n6k3}
  Let $\Lambda$ be the primitive lattice in $n = 6$ dimensions spanned by the $k = 3$ vectors
  \[
    \begin{bmatrix}
      1 \\ 0 \\ 0 \\ 0 \\ 0 \\ 0
    \end{bmatrix},\quad
    \begin{bmatrix}
      0 \\ 1 \\ T \\ 0 \\ 0 \\ 0
    \end{bmatrix},\quad
    \begin{bmatrix}
      0 \\ 0 \\ 0 \\ 1 \\ T \\ T^2
    \end{bmatrix}.
  \]
  Observe that any matrix $A$ of the shape
  \[
    A = \left[
      \begin{array}{c|cc|ccc}
          a & b & c & d & e & f \\ \hline
            & g &   & h & i &   \\
            &   & g &   & h & i \\ \hline
            &   &   & j &   &   \\
            &   &   &   & j &   \\
            &   &   &   &   & j
      \end{array}
    \right]
  \]
  fixes $\Lambda$. Moreover, each choice of integers $a, b, \ldots, j$ in the range $[0, T)$ yields a different choice of the corresponding restriction
  \[
    G = \begin{bmatrix}
        a & b + Tc & d + Te + T^2 f \\
          & g      & h + Ti         \\
          &        & j
    \end{bmatrix}.
  \]
  So there are at least $T^{10}$ (in fact, $\asymp T^{10}$) matrices $G$, more than the $\asymp T^{k^2} = T^9$ that appear for $\Lambda = \<e_1,e_2,e_3\>$.
\end{examp}
Generalizing this example appropriately, we find that there is a lattice $\Lambda$ for $n = \binom{k + 1}{2}$ with $\gg T^{\binom{k + 2}{3}}$ choices for $G$. The exponent $\binom{k + 2}{3}$ is of higher order than $k^2$, and indeed is smaller only by a constant factor than the bound $\#G \ll T^{k n}$, which arises by noting that $G$ is determined by the $k$ rows of $A$ describing the projection of outputs of $A$ onto any coordinate $k$-plane whose projection to $W$ is faithful.

\subsection{Counting \texorpdfstring{$G$}{G}}
Our first lemma limits the possibilities for $G$.
\begin{lem}
  The matrix $G$ as above has entries
  \begin{equation} \label{eq:bound_g}
    g_{ij} \ll \frac{T \ell_j}{\ell_i}.
  \end{equation}
\end{lem}
\begin{proof}
  A reduced basis $v_1,\ldots,v_k$ for a lattice $\Lambda$ has the useful property that, for any decomposition of a vector $v = a_1 v_1 + \cdots + a_k v_k$, we have
  \[
    \norm{a_i v_i} \ll \norm{v},
  \]
  where the implied constant depends only on $k$. Thus, we can write
  \begin{align*}
    \size{g_{ij}} &= \frac{\norm{g_{ij} v_i}}{\norm{v_i}} \\
    &\ll \frac{\Norm{\sum_{h} g_{hj} v_h}}{\norm{v_i}} \\
    &= \frac{\Norm{A v_j}}{\Norm{v_i}} \\
    &\ll \frac{\Norm{A} \Norm{v_j}}{\Norm{v_i}} \\
    &\ll \frac{T \ell_j}{\ell_i},
  \end{align*}
  proving the assertion.
\end{proof}
\begin{lem}
  The number of possible $G$, for a given $\Lambda$, is $\ll T^{kn - n + 1}$.
\end{lem}
\begin{proof}
  Let $c$ be the implied constant in \eqref{eq:bound_g}, so that $|g_{ij}| \leq c T \ell_j / \ell_i$. Note that if $c T \ell_i / \ell_j < 1$, then $g_{ij} = 0$; in this case, we say that $(i,j)$ is a \emph{pinch point}. For instance, in Example \ref{ex:n6k3}, there is a pinch point at $(3,1)$. Let $P \subseteq [n]^2$ be the set of pinch points. If $(i,j)$ is not a pinch point, the number of possibilities for $g_{ij}$ is $\ll T \ell_j / \ell_i$. Multiplying,
  \begin{align}
      \# G &\ll \prod_{(i,j) \notin P} \frac{T \ell_j}{\ell_i} \nonumber \\
      &= \prod_{i,j = 1}^k \frac{T \ell_j}{\ell_i} \prod_{(i,j) \in P} \frac{\ell_i}{T \ell_j} \nonumber \\
      &= T^{k^2} \prod_{(i,j) \in P} \frac{\ell_i}{T \ell_j}. \label{eq:pinch_prod}
  \end{align}
  
  Since $\Lambda$ is $k$-dimensional, the projection of $\Lambda$ onto some $k$-dimensional coordinate subspace, WLOG the first $k$ coordinates, is full-dimensional. Then the first $k$ rows of $A$ determine $G$. Moreover, each pinch point imposes a linear condition on the entries in the first $k$ rows of $A$. Hence $G$ is determined by $kn - \size{P}$ entries of $A$, each of which is $\ll T$, so we get the bound
  \begin{equation} \label{eq:G_first_way}
    \# G \ll T^{kn - \size{P}}.
  \end{equation}
  If $\size{P} \geq n - 1$, then \eqref{eq:G_first_way} proves the lemma. Hence we may assume that $\size{P} \leq n - 2$.

  Note that if $(i,j)$ is a pinch point, $i' \geq i$, and $j' \leq j$, then $(i',j')$ is also a pinch point. In other words, the pinch points form a lower left corner of the $k \times k$ array. Clearly there are no pinch points $(i,j)$ with $i \leq j$. Also, $(i + 1, i)$ is never a pinch point: if it is, it forces the entire lower left $i \times (k - i)$ submatrix of $G$ to be zero, so that $G$ is block upper triangular, contradicting our hypothesis that its characteristic polynomial $g$ is irreducible! So we have the bound
  \begin{equation} \label{eq:bound_l}
    \ell_{i+1} \ll T \ell_i.
  \end{equation}
  Let row $i$ have $p_i$ pinch points $(i,1),\ldots,(i,p_i)$. For $(i,j) \in P$, each factor in \eqref{eq:pinch_prod} is bounded by
  \[
    \frac{\ell_i}{T \ell_j} \ll T^{p_i + 1 - j} \frac{\ell_i}{T \ell_{p_i + 1}} \ll T^{p_i + 1 - j},
  \]
  yielding a bound $\#G \ll T^r$ where
  \begin{align*}
    r &= k^2 + \sum_{i = 3}^k \sum_{j = 1}^{p_i} (p_i + 1 - j) \\
    &= k^2 + \sum_{i = 3}^k \frac{p_i(p_i + 1)}{2} \\
    &\leq k^2 + \sum_{i = 3}^k \frac{p_i(k - 1)}{2} \\
    &= k^2 + \frac{k - 1}{2} \sum_{i = 3}^k p_i \\
    &= k^2 + \frac{k - 1}{2} |P| \\
    &\leq k^2 + \frac{k - 1}{2} (n-2) \\
    &\leq k^2 + \frac{k - 1}{2} (2n - 2k - 2) \\
    &= nk - n + 1. \qedhere
  \end{align*}
\end{proof}

\subsection{Counting \texorpdfstring{$A$}{A} given \texorpdfstring{$G$}{G}}

\begin{proof}[Proof of Theorem \ref{thm:red_short}]
If $\Lambda$ and $G$ are fixed, then $A$ is determined up to translation by the lattice $\Lambda^\perp \tensor \ZZ^n$ of matrices each of whose rows annihilates $\Lambda$. If $\Lambda^\perp$ is generated by vectors of length $\leq T$, we get
\begin{align*}
  \#A &\ll \#G \cdot \frac{T^{n(n - k)}}{d(\Lambda^\perp \tensor \ZZ^n)} \\
  &= \#G \cdot \frac{T^{n(n-k)}}{d(\Lambda)^n} \\
  &\ll \frac{T^{n^2 - n + 1}}{d(\Lambda)^n}.
\end{align*}
The result follows by summing over $\Lambda$ with $d(\Lambda) \ll T^{O(1)}$ (by \eqref{eq:Lambda_bounded_T_O1}) and using the bound $\sum_{d(\Lambda) \ll T^{O(1)}} d(\Lambda)^{-n} = O(\log T)$, which Katznelson deduces from a lattice-counting estimate due to Schmidt \cite[equation (11)]{Kat93}.
\end{proof}

Note that the restriction on the length of the vectors (bounded by $T$) is closely related to the difficult ``large value problem" \cite{Guth} that appears in multiple fields of mathematics and computer science.  Essentially, this problem asks: for a given matrix $A$ and bounded input vector $v$, how many entries of $Av$ can be large?  The lack of control of large entries of $\Lambda^\perp$ leads to our restriction on the length of its generating vectors, which then leads to the overall savings from counting these matrices.  It is possible that developments in large value problems could be useful in similar matrix counting problems.
\section{Primitive groups with large field discriminant and Theorem \ref{thm:geosieve}} \label{sec:geosieve}

In this section, we bound the number $L_n(T)$ of matrices such that $\chi_a$ defines a primitive number field $K_A$ of discriminant $D = \disc K_A \geq T^2$. Our methods are direct adaptations of Bhargava's.

Let $C = \prod_{p\mid D} p$. First, we note that $D$ is squarefull.

\begin{lem}[\cite{Bhargava_vdW}, remark following Proposition 11]
  If a prime $p$ divides $C$, then $p^2 \mid D$.
\end{lem}

Recall \cite[Proposition 21]{Bhargava_vdW} that if $f \in \FF_p[x]$ is a polynomial of degree $n$, we factor $f = \prod_i f_i^{e_i}$ into powers of distinct irreducibles and define
the \emph{index} of $f$ to be
\[
  \ind f = \sum_i (e_i - 1) \deg f_i = n - \sum_i \deg f_i.
\]
The index has the following useful properties:
\begin{lem}[see \cite{ABO2024}, Lemma 5.3]
  If $p^k \mid D$, then the index of $\chi_A$ mod $p$ is at least $k$.
\end{lem}
\begin{lem}[see \cite{Bhargava_vdW}, Proposition 21]
  \label{lem:index_rare}
  The fraction of polynomials modulo $p$ having index at least $k$ is $O(p^{-k})$.
\end{lem}

We need the approximate equidistribution of characteristic polynomials of matrices over $\FF_p$. Indeed, an exact formula for the number of matrices with a given characteristic polynomial is given in \cite{Reiner} and \cite{Gerstenhaber}:

\begin{prop}[Reiner \cite{Reiner}, Theorem 2, and Gerstenhaber \cite{Gerstenhaber}]
\label{prop:char_poly_eqdst}
  For any monic $f \in \FF_q[x]$ of degree $n$ over a finite field $\FF_q$, the number of matrices $A \in \Mat_n(\FF_q)$ having characteristic polynomial $f$ is
  \[
    q^{n^2 - n}\bigg(1 + O\Big(\frac{1}{q}\Big)\bigg).
  \]
  (Here the implied constant depends on $n$, but not $q$ or $f$.)
\end{prop}

Observe that Lemma \ref{lem:index_rare} and Proposition \ref{prop:char_poly_eqdst} together imply that the fraction of matrices $A \in \Mat_n \FF_p$ with $\ind (\chi_A \bmod p) \geq k$ is $O(p^{-k})$. We are now ready to prove Theorem \ref{thm:geosieve}.

\begin{proof}[Proof of Theorem \ref{thm:geosieve}]

We separate the cases $C \leq T$ and $C > T$, corresponding to Bhargava's Cases I and III respectively.

If $C \leq T$, then the $A$ having any particular value of $D$ satisfy congruence conditions modulo $C$ for which the solution set has density $O(1/D)$. Therefore, there are at most $O(T^{n^2}/D)$ such $A$. Summing, the total contribution to $L_n(T)$ is
\[
\ll \sum_{D \geq T^2} \frac{T^{n^2}}{D} \ll T^{n^2 - 1}.
\]

We now assume $C \geq T$. Let $S = \Diag(1,2,\ldots, n)$; note that $\chi_S$ is separable modulo every prime $p \geq n$. For a matrix $A$, define the \emph{double discriminant} of $A$ to be
\[
  \DD(A) = \disc_{t} \disc_x \chi_{A + tS}(x) = \disc_t \disc_x \det(xI - tS - A).
\]
Note that $\DD(A)$ is an integer polynomial in the entries of $A$ that is invariant under translation by multiples of $S$. We can check that $\DD(A)$ is not identically zero by observing that for a general diagonal matrix $A$, there are the full $\binom{n}{2} = \deg_t \disc_x(xI - tS - A)$ values of $t$ that cause two of the diagonal entries of $A + tS$ to coincide.

If $p\mid C$, so that $\ind (\chi_A \bmod p) \geq 2$, we claim that $p \mid \DD(A)$. To see this, observe that the factorization constraint on $\chi_{A\bmod p}$ (a root of multiplicity at least $3$, or at least two double roots) makes $\chi_{A\bmod p}$ a singular point of the hypersurface of polynomials of discriminant $0$. Thus, in whatever direction $A$ is perturbed, $\disc \chi_{A\bmod p}$ will vanish to order at least $2$, making $\DD(A\bmod p) = 0$ (see also Proposition 4 in \cite{Bhargava_vdW_short}).

Therefore, the $A$ having a given $D$ are constrained by the relations
\begin{equation} \label{eq:DDmodC}
\disc \chi_A \equiv \DD(A) \equiv 0 \mod C.
\end{equation}

There are $O(T^{n^2 - 1})$ congruence classes of matrices $A$ modulo $\ZZ S$, because the $n-1$ linear combinations of diagonal entries
\[
  2a_{11} - a_{22}, \quad 3a_{11} - a_{33}, \quad \ldots, \quad n a_{11} - a_{nn},
\]
together with the $n^2 - n$ off-diagonal entries, determine $A$ up to $\ZZ S$ and are $O(T)$ in size. Also, in each congruence class there are $O(T)$ matrices of norm $\leq T$ because $a_{11}$ determines $A$.  Since $\DD(A)$ is a nonzero polynomial, there are $O(T^{n^2 - 1})$ matrices $A$ lying in congruence classes on which $\DD(A) = 0$. Excluding these, when $A = A_0 + tS$ is in a fixed congruence class, there are at most $2^{\omega(\DD(A))} \ll_\epsilon T^\epsilon$ choices of $C > T$ that divide $\DD(A)$. Then $t$ is constrained by
\begin{equation}\label{eq:last_entry}
  \disc \chi_A = \disc_x \det(xI - tS - A) \equiv 0 \mod C.
\end{equation}
We claim that \eqref{eq:last_entry} is a nonzero polynomial equation in $t$ modulo every prime $p \mid C$ with $p \geq n$. For this, note that $\det(xI - tS - A)$ is a polynomial of degree $n$ in $x$ and $t$ whose leading homogeneous piece is
\[
  \det(xI - tS) = \chi_{t S}(x) = (x - t)(x - 2 t)\cdots (x - n t).
\]
Hence $\disc_x \det(xI - tS - A)$ has leading term $t^{n^2 - n} \disc \chi_S \nequiv 0 \mod p$ and in particular is not identically $0$ modulo $p$.

Now \eqref{eq:last_entry} has $O(1)$ solutions modulo every prime $p \mid C$, and so $O(1)^{\omega(C)} \ll_\epsilon T^\epsilon$ solutions overall. Hence the total number of matrices $A$ in this case is $O_\epsilon(T^{n^2 - 1} \cdot T^\epsilon \cdot T^\epsilon) = O_\epsilon(T^{n^2 - 1 + \epsilon})$, as desired.
\end{proof}

\bibliography{suppl,matrix}

\newcommand{\etalchar}[1]{$^{#1}$}
\begin{thebibliography}{AGLO{\etalchar{+}}23}

\bibitem[ABO24]{ABO2024}
Theresa~C. Anderson, Adam Bertelli, and Evan~M. O'Dorney.
\newblock Galois groups of reciprocal polynomials and the van der
  {W}aerden--{B}hargava theorem, 2024.
\newblock Submitted.
  \href{arxiv.org/abs/2406.18970}{\texttt{arxiv:\allowbreak2406.\allowbreak18970}}.

\bibitem[AGLO{\etalchar{+}}23]{6author}
Theresa~C. Anderson, Ayla Gafni, Robert~J. Lemke~Oliver, David Lowry-Duda,
  George Shakan, and Ruixiang Zhang.
\newblock Quantitative {H}ilbert irreducibility and almost prime values of
  polynomial discriminants.
\newblock {\em Int. Math. Res. Not. IMRN}, 2023(3):2188--2214, 2023.

\bibitem[BGS25]{2025arXiv250217218B}
Lior {Bary-Soroker}, Daniele {Garzoni}, and Sasha {Sodin}.
\newblock {Irreducibility of the characteristic polynomials of random
  tridiagonal matrices}.
\newblock {\em arXiv e-prints}, page arXiv:2502.17218, February 2025.

\bibitem[Bha23]{Bhargava_vdW_short}
Manjul Bhargava.
\newblock A proof of van der {W}aerden's conjecture on random {G}alois groups
  of polynomials.
\newblock {\em Pure Appl. Math. Q.}, 19(1):45--60, 2023.

\bibitem[Bha25]{Bhargava_vdW}
Manjul Bhargava.
\newblock Galois groups of random integer polynomials and van der {W}aerden's
  conjecture.
\newblock {\em Ann. of Math. (2)}, 201(2):339--377, March 2025.

\bibitem[BL24]{BlomerLutsko2024}
Valentin Blomer and Christopher Lutsko.
\newblock Hyperbolic lattice point counting in unbounded rank.
\newblock {\em J. Reine Angew. Math.}, 2024(812):257--274, 2024.

\bibitem[BSK20]{BSK20}
Lior Bary-Soroker and Gady Kozma.
\newblock Irreducible polynomials of bounded height.
\newblock {\em Duke Math. J.}, 169(4):579--598, 2020.

\bibitem[BSKK23]{BSKK23}
Lior Bary-Soroker, Dimitris Koukoulopoulos, and Gady Kozma.
\newblock Irreducibility of random polynomials: general measures.
\newblock {\em Invent. Math.}, 233(3):1041--1120, 2023.

\bibitem[BV19]{BV19}
Emmanuel Breuillard and P\'{e}ter~P. Varj\'{u}.
\newblock Irreducibility of random polynomials of large degree.
\newblock {\em Acta Math.}, 223(2):195--249, 2019.

\bibitem[CD20]{CD2020}
Sam Chow and Rainer Dietmann.
\newblock Enumerative {G}alois theory for cubics and quartics.
\newblock {\em Adv. Math.}, 372:107282, 37, 2020.

\bibitem[Coh79]{Coh79}
S.~D. Cohen.
\newblock The distribution of the {G}alois groups of integral polynomials.
\newblock {\em Illinois J. Math.}, 23(1):135--152, 1979.

\bibitem[Ebe22]{Eberhard}
Sean Eberhard.
\newblock The characteristic polynomial of a random matrix.
\newblock {\em Combinatorica}, 42(4):491--527, 2022.

\bibitem[EMS96]{EMS96}
Alex Eskin, Shahar Mozes, and Nimish Shah.
\newblock Unipotent flows and counting lattice points on homogeneous varieties.
\newblock {\em Ann. of Math. (2)}, 143(2):253--299, 1996.

\bibitem[FJSS23]{FJSS}
Asaf Ferber, Vishesh Jain, Ashwin Sah, and Mehtaab Sawhney.
\newblock Random symmetric matrices: rank distribution and irreducibility of
  the characteristic polynomial.
\newblock {\em Math. Proc. Cambridge Philos. Soc.}, 174(2):233--246, 2023.

\bibitem[Gal73]{Gallagher1972}
P.~X. Gallagher.
\newblock The large sieve and probabilistic {G}alois theory.
\newblock In {\em Analytic number theory ({P}roc. {S}ympos. {P}ure {M}ath.,
  {V}ol. {XXIV}, {S}t. {L}ouis {U}niv., {S}t. {L}ouis, {M}o., 1972)}, Proc.
  Sympos. Pure Math., Vol. XXIV, pages 91--101. Amer. Math. Soc., Providence,
  RI, 1973.

\bibitem[Ger61]{Gerstenhaber}
Murray Gerstenhaber.
\newblock On the number of nilpotent matrices with coefficients in a finite
  field.
\newblock {\em Illinois J. Math.}, 5:330--333, 1961.

\bibitem[Gut25]{Guth}
Larry Guth.
\newblock Large value estimates in number theory, harmonic analysis, and
  computer science, 2025.
\newblock Preprint,
  \href{https://arxiv.org/abs/2503.07410}{\texttt{arxiv:\allowbreak2503\allowbreak.07410}}.

\bibitem[Hok23]{Hokken23}
David Hokken.
\newblock Counting (skew-)reciprocal {L}ittlewood polynomials with square
  discriminant, 2023.
\newblock Preprint, \href{arxiv.org/abs/2301.05656}{\texttt{arxiv:\allowbreak
  2301.\allowbreak 05656}}.

\bibitem[HOS24]{ShpI}
Philipp Habegger, Alina Ostafe, and Igor~E. Shparlinski.
\newblock Integer matrices with a given characteristic polynomial and
  multiplicative dependence of matrices, 2024.
\newblock Preprint,
  \href{https://arxiv.org/abs/2203.03880}{\texttt{arxiv:\allowbreak2203.\allowbreak03880}}.

\bibitem[Kat93]{Kat93}
Y.~R. Katznelson.
\newblock Singular matrices and a uniform bound for congruence groups of {${\rm
  SL}_n({\bf Z})$}.
\newblock {\em Duke Math. J.}, 69(1):121--136, 1993.

\bibitem[Kno55]{Knobloch1955}
Hans-Wilhelm Knobloch.
\newblock Zum {H}ilbertschen {I}rreduzibilit\"{a}tssatz.
\newblock {\em Abh. Math. Sem. Univ. Hamburg}, 19:176--190, 1955.

\bibitem[LOT22]{LOT20}
Robert~J. Lemke~Oliver and Frank Thorne.
\newblock Upper bounds on number fields of given degree and bounded
  discriminant.
\newblock {\em Duke Math. J.}, 171(15):3077--3087, 2022.

\bibitem[OS24]{OSNonDiag}
Alina Ostafe and Igor~E. Shparlinski.
\newblock On the sparsity of non-diagonalisable integer matrices and matrices
  with a given discriminant, 2024.
\newblock Preprint,
  \href{https://arxiv.org/abs/2312.12626}{\texttt{arxiv:\allowbreak2312\allowbreak.12626}}.

\bibitem[Rei61]{Reiner}
Irving Reiner.
\newblock On the number of matrices with given characteristic polynomial.
\newblock {\em Illinois J. Math.}, 5:324--329, 1961.

\bibitem[Riv06]{Rivin06pre}
Igor Rivin.
\newblock Counting reducible matrices, polynomials, and surface and free group
  automorphisms, 2006.
\newblock Preliminary version of \cite{Rivin08}.
  \href{https://arxiv.org/abs/math/0604489}{\texttt{arxiv:math/0604489}}.

\bibitem[Riv08]{Rivin08}
Igor Rivin.
\newblock Walks on groups, counting reducible matrices, polynomials, and
  surface and free group automorphisms.
\newblock {\em Duke Math. J.}, 142(2):353--379, 2008.

\bibitem[Sch95]{Schmidt}
Wolfgang~M. Schmidt.
\newblock Number fields of given degree and bounded discriminant.
\newblock {\em Ast\'{e}risque}, 4(228):189--195, 1995.
\newblock Columbia University Number Theory Seminar (New York, 1992).

\bibitem[Ser13]{SerreMW}
Jean-Pierre Serre.
\newblock {\em Lectures on the {M}ordell-{W}eil Theorem}.
\newblock Aspects of Mathematics. Vieweg+Teubner Verlag, Wiesbaden, 3rd
  edition, 2013.

\bibitem[Shp10]{Shp09}
I.~E. Shparlinski.
\newblock Some counting questions for matrices with restricted entries.
\newblock {\em Linear Algebra Appl.}, 432(1):155--160, 2010.

\bibitem[vdW36]{vdW1936}
B.~L. van~der Waerden.
\newblock Die {S}eltenheit der reduziblen {G}leichungen und der {G}leichungen
  mit {A}ffekt.
\newblock {\em Monatsh. Math. Phys.}, 43(1):133--147, 1936.

\bibitem[{Zyw}10]{Zywina}
David {Zywina}.
\newblock {Hilbert's irreducibility theorem and the larger sieve}.
\newblock {\em arXiv e-prints}, page arXiv:1011.6465, November 2010.

\end{thebibliography}
\bibliographystyle{alpha}

\end{document}